\documentclass{amsart}
\usepackage{graphicx}
\usepackage[all]{xy}
\usepackage{amscd}
\usepackage{amssymb}
\usepackage{amsmath}
\usepackage{amsthm}
\usepackage{amsfonts}
\usepackage{graphics}
\usepackage{epsfig,psfrag}
\usepackage[english]{babel}
\usepackage{latexsym}
\usepackage{mathrsfs}
\usepackage{hyperref}

\newtheorem{theorem}{Theorem}[section]

\newtheorem{lemma}[theorem]{Lemma}

\newtheorem{proposition}[theorem]{Proposition}
\newtheorem{question}[theorem]{Question}

\theoremstyle{definition}
\newtheorem{example}[theorem]{Example}
\newtheorem{definition}[theorem]{Definition}

\theoremstyle{remark}
\newtheorem{remark}[theorem]{Remark}

\begin{document}

\title[Extending periodic automorphisms of surfaces to 3-manifolds]
{Extending periodic automorphisms of surfaces to 3-manifolds}

\author{Yi Ni}
\address{Department of Mathematics, California Institute
of Technology, Pasadena 91125, USA}
\email{yini@caltech.edu}

\author{Chao Wang}
\address{School of Mathematical Sciences \& Shanghai Key Laboratory of PMMP, East China Normal University, Shanghai 200241, CHINA}
\email{chao\_{}wang\_{}1987@126.com}

\author{Shicheng Wang}
\address{School of Mathematical Sciences, Peking University, Beijing 100871, CHINA}
\email{wangsc@math.pku.edu.cn}



\thanks{We thank Francis Bonahon, Jianfeng Lin, Yimu Zhang and Bruno Zimmermann for useful communications. The first author is supported by NSF grant
numbers DMS-1252992 and DMS-1811900.
The second author is supported by Science and Technology Commission of Shanghai Municipality (STCSM), grant No. 18dz2271000.
The third author is supported by NSF of China grant No. 11771021.}

\begin{abstract}
Let $G$ be a finite group acting on a connected compact surface $\Sigma$, and $M$ be an integer homology 3-sphere. We show that if each element of $G$ is extendable over $M$ with respect to a fixed embedding $\Sigma\rightarrow M$, then $G$ is extendable over some $M'$ which is 1-dominated by $M$. From this result, in the orientable category we classify all periodic automorphisms of closed surfaces that are extendable over the 3-sphere. The corresponding embedded surface of such an automorphism can always be a Heegaard surface.
\end{abstract}

\date{}
\maketitle

\section{Introduction}
We work in the smooth category, mainly consider oriented manifolds, and use $Aut(\cdot)$ to denote the orientation-preserving automorphism group of a manifold.

Let $\Sigma$ be a compact oriented surface and $M$ be an oriented 3-manifold, where $\Sigma$ and $M$ are possibly disconnected. An element $f$ in $Aut(\Sigma)$ is {\it extendable} over $M$ with respect to an embedding $e:\Sigma\rightarrow M$ if there exists an element $f'$ in $Aut(M)$ such that $f'\circ e=e\circ f$. A subgroup $G$ in $Aut(\Sigma)$ is {\it extendable} over $M$ with respect to an embedding $e:\Sigma\rightarrow M$ if there exists a group monomorphism $\phi:G\rightarrow Aut(M)$ such that $\phi(h)\circ e=e\circ h$ for any $h\in G$. If $\Sigma\subset M$, then the embedding is the inclusion map when we talk about ``extendable''.

Let $\Sigma_g$ be a closed connected oriented surface of genus $g$. We are interested in the following question, where the most interesting and basic case is when $M$ is the 3-sphere $S^3$.

\begin{question}\label{qes:main}
How to classify the finite subgroup in $Aut(\Sigma_g)$ such that each of its element is extendable over $M$ with respect to a fixed embedding $\Sigma_g\rightarrow M$?
\end{question}

For closed connected oriented 3-manifolds $M$ and $M'$, $M'$ is {\it 1-dominated} by $M$, denoted by $M\succeq_1 M'$, if there exists a degree one map $M\rightarrow M'$.

For integer homology 3-spheres, we have the following result.

\begin{theorem}\label{thm:general case}
Given an integer homology 3-sphere $M$ and a finite subgroup $G$ in $Aut(\Sigma_g)$, if there is an embedding $e:\Sigma_g\rightarrow M$ such that each element of $G$ is extendable over $M$ with respect to $e$, then there exists an integer homology 3-sphere $M'$ such that $M\succeq_1 M'$ and $G$ is extendable over $M'$.
\end{theorem}

Note that on the set of integer homology 3-spheres $\succeq_1$ is a partial order relation (see Proposition \ref{pro:degree one}), and for a given $M$ there are only finitely many $M'$ satisfying $M\succeq_1 M'$ (see \cite{BRW} and \cite{Liu}). As an example, when $M$ is the Poincar\'e homology sphere $S^3_P$, Question \ref{qes:main} reduces to classifying the finite subgroup that is extendable over $S^3_P$ or $S^3$. Especially, when $M$ is $S^3$, we have the following result.

\begin{theorem}\label{thm:sphere case}
A finite subgroup $G$ in $Aut(\Sigma_g)$ is extendable over $S^3$ if and only if there is an embedding $e:\Sigma_g\rightarrow S^3$ such that each element of $G$ is extendable over $S^3$ with respect to $e$.
\end{theorem}

Because every element in $Aut(S^3)$ is isotopic to the identity (see \cite{Ha}) and every finite subgroup in $Aut(S^3)$ can be conjugated into $SO(4)$ (see \cite{Pe}), Theorem \ref{thm:sphere case} means that if a finite subgroup in $Aut(\Sigma_g)$ can be realized by topological motions of $S^3$ (or $\mathbb{R}^3$), then it can be realized by isometric motions of $S^3$. Note that finite subgroups in $Aut(\Sigma_g)$ that are extendable over $S^3$ with order bigger than $4(g-1)$ can be classified (see \cite{WWZZ2}). The following theorem classifies the finite cyclic subgroups in $Aut(\Sigma_g)$ that are extendable over $S^3$ (see also Remark \ref{rem:equivalent}).

Given a finite subgroup $G$ in $Aut(\Sigma_g)$, we can obtain an orbifold $\Sigma_g/G$. Suppose that it has underlying space $\Sigma_r$ and $s$ singular points of indices $n_1,\cdots,n_s$. Then the fundamental group $\pi_1(\Sigma_g/G)$ has a presentation (see Figure \ref{fig:presentation} in Lemma \ref{lem:orb group})
\[\langle \alpha_1,\beta_1,\cdots,\alpha_r,\beta_r,\gamma_1,\cdots,\gamma_s\mid \prod^r_{i=1}[\alpha_i,\beta_i]\prod^{s}_{j=1}\gamma_j=1,\gamma_k^{n_k}=1,1\leq k\leq s\rangle.\]
The $G$-action on $\Sigma_g$ gives an epimorphism $\psi:\pi_1(\Sigma_g/G)\rightarrow G$ which is injective on finite subgroups of $\pi_1(\Sigma_g/G)$. We call such $\psi$ a {\it finitely-injective} epimorphism.

\begin{theorem}\label{thm:cyclic classify}
A finite cyclic subgroup $G$ in $Aut(\Sigma_g)$ is extendable over $S^3$ if and only if the orbifold $\Sigma_g/G$ and the finitely-injective epimorphism $\psi$ satisfy:

(a) there exist co-prime positive integers $p,q$ such that $n_1,\cdots,n_s\in\{p,q\}$;

(b) if $n_i=n_j$ for some $i\neq j$, then either $\psi(\gamma_i)=\psi(\gamma_j)$ or $\psi(\gamma_i\gamma_j)=1$;

(c) $\gamma_1,\cdots,\gamma_s$ can be partitioned into pairs $\gamma_i,\gamma_j$ such that $\psi(\gamma_i\gamma_j)=1$.

Two such subgroups $G$ and $G'$ are conjugate in $Aut(\Sigma_g)$ if and only if they are isomorphic and $\Sigma_g/G$ and $\Sigma_g/{G'}$ are homeomorphic. Moreover, the corresponding embedded surface of such a subgroup $G$ can always be a Heegaard surface.
\end{theorem}

Actually, if such a $G$ does not give free actions on $\Sigma_1$, then its conjugate class is determined by $\Sigma_g/G$, which can be enumerated by the Riemann-Hurwitz formula. We also has a standard form of the $G$-action (see Example \ref{exa:standard form}). As a comparison, there exist finite subgroups in $Aut(\Sigma_{21})$ and $Aut(\Sigma_{481})$ which are extendable over $S^3$, but the embedded surfaces cannot be Heegaard surfaces (see \cite{WWZZ2}). It is worth to mention that  if a homeomorphism of $\Sigma_2$ is extendable over $S^3$, then its corresponding embedded surface can always be a Heegaard surface \cite{FK}.

After giving some preparations in Section~\ref{sec:partial order}, we will prove a stronger version of Theorem~\ref{thm:general case} in Section~\ref{sec:extending} and generalize the result to general actions on compact manifolds. Then, in Section~\ref{sec:cyclic}, we will prove Theorem~\ref{thm:cyclic classify} and give some intuitive examples which can also be read directly after the introduction. As the end of the introduction, we mention some literatures related to Question~\ref{qes:main}:

1. There are many results about extending finite group actions on $\Sigma_g$ to some 3-manifold bounded by $\Sigma_g$. For example, finite cyclic group actions are analyzed in the pioneer work \cite{Bo1}, and finite abelian group actions are analyzed in \cite{RZ}.

2. A result similar to our Theorem \ref{thm:sphere case} has been obtained in \cite{Fl} for finite cyclic group actions on finite 3-connected graphs. In this direction, the recent results in \cite{FY} are close to the style of our Theorem \ref{thm:general case}.

3. To get an intuition about the symmetries on surfaces, a sequence of papers on embedding symmetries of $\Sigma$ into those of $S^3$
appeared recently, including \cite {WWZZ1}, \cite{WWZZ2} and \cite{GWWZ}. The first two devote to maximum order problems,
and the third one lists all extendable finite cyclic group actions on $\Sigma_2$. The present research is inspired by those papers.

\section{Relation ``$\succeq$'' and property ``$\Sigma_g$-splittable''}\label{sec:partial order}
In this section, we introduce two concepts that will be used in our main result Theorem \ref{thm:stronger}. Let $\mathcal{M}$ denote the set of closed connected oriented 3-manifolds. We first define the relation $\succeq$ on $\mathcal{M}$ and give some properties about this relation. Then we define $\Sigma_g$-splittable for manifolds in $\mathcal{M}$.

For $M$ in $\mathcal{M}$ and a compact connected 3-manifold $N$ embedded in $M$, let $\overline{M-N}$ denote the closure of $M-N$, and let $\partial N$ denote the common boundary of $N$ and $\overline{M-N}$. If $\partial N$ is a torus, then up to isotopy there exists a unique simple closed curve $c$ in $\partial N$ such that $[c]\neq 0$ in $H_1(\partial N,\mathbb{Q})$ but $[c]=0$ in $H_1(\overline{M-N},\mathbb{Q})$, by the following ``half-live half-die'' lemma. Hence, we can obtain a closed 3-manifold $M'$ from $N$ by filling a solid torus into $\partial N$, mapping the meridian to $c$.

\begin{lemma}\label{lem:h-l h-d}
For a compact orientable 3-manifold $X$, the dimension of the kernel of $H_1(\partial X,\mathbb{Q})\rightarrow H_1(X,\mathbb{Q})$ is half of the dimension of $H_1(\partial X,\mathbb{Q})$.
\end{lemma}

\begin{definition}\label{def:partial order}
For $M$ and $M'$ in $\mathcal{M}$, if $M'$ can be obtained from $M$ as in the above construction, then define the relation $M\succeq_s M'$ (``s'' for ``surgery''). All the relations ``$\succeq_s$'' generate a reflexive and transitive relation on $\mathcal{M}$, denoted by ``$\succeq$''.
\end{definition}

Note that if $M'$ can be obtained from $M$ by pinching an embedded compacted 3-manifold whose boundary is a sphere or a torus to a ball or a solid torus, then $M\succeq_s M'$. We also have another relation $\succeq_H$ on $\mathcal{M}$, where $M\succeq_H M'$ if and only if there is an epimorphism $H_1(M,\mathbb{Q})\rightarrow H_1(M',\mathbb{Q})$. It is well known that $M\succeq_1 M'$ implies $M\succeq_H M'$.

\begin{proposition}\label{pro:order imply}
For $M$ and $M'$ in $\mathcal{M}$, $M\succeq M'$ implies $M\succeq_H M'$; if $M$ is an integer homology 3-sphere, then $M\succeq M'$ implies $M\succeq_1 M'$ and $M'$ is also an integer homology 3-sphere.
\end{proposition}

\begin{proof}
Suppose that $M\succeq_s M'$, $N$ is the compact 3-manifold embedded in $M$, and $c$ is the simple closed curve in $\partial N$. By the MV-sequence of homology groups,
\begin{align*}
\dim H_1(M,\mathbb{Q})=&\dim H_1(N,\mathbb{Q})+\dim H_1(\overline{M-N},\mathbb{Q})\\
&-\dim Im(H_1(\partial N,\mathbb{Q})\rightarrow H_1(N,\mathbb{Q})\oplus H_1(\overline{M-N},\mathbb{Q})),
\end{align*}
where ``$Im(\cdot)$'' denotes the image of the map. There is a similar equality for $M'$ with $\overline{M-N}$ replaced by the solid torus which is filled into $\partial N$ along $c$.

Let $c'$ be a simple closed curve in $\partial N$ such that $[c]$ and $[c']$ generate $H_1(\partial N,\mathbb{Q})$. For $M$ and $M'$, since $[c]=0$ in $H_1(\overline{M-N};\mathbb Q)$ and $H_1(\overline{M'-N};\mathbb Q)$, the images of $[c]$ belong to $H_1(N,\mathbb{Q})$. By Lemma~\ref{lem:h-l h-d}, the images of $[c']$ are nonzero in the second factors, respectively. Hence, for $M$ and $M'$ the last summands are equal, and
\[\dim H_1(M,\mathbb{Q})-\dim H_1(M',\mathbb{Q})=\dim H_1(\overline{M-N},\mathbb{Q})-1\geq 0.\]

The general case when $M\succeq M'$ can be obtained by induction, so $M\succeq_H M'$.

If $M$ is an integer homology 3-sphere and $M\succeq_s M'$, then by the MV-sequence both $N$ and $\overline{M-N}$ are integer homology solid tori and $[c]=0$ in $H_1(\overline{M-N},\mathbb{Z})$. Then, $c$ bounds a surface in $\overline{M-N}$ and there is a degree one map from $\overline{M-N}$ to a solid torus, which is identity on $\partial N$ and maps the surface to a meridian disk. Combined with the identity on $N$, we get a degree one map from $M$ to $M'$. Then, $H_1(M',\mathbb{Z})=0$ and $M'$ is also an integer homology 3-sphere.

The general case when $M\succeq M'$ can be obtained by induction since on each step we have an integer homology 3-sphere.
\end{proof}

\begin{proposition}\label{pro:degree one}
On the set of integer homology 3-spheres the relation $\succeq_1$ is a partial order relation, hence $\succeq$ is also a partial order relation.
\end{proposition}

\begin{proof}
We need to show that $\succeq_1$ is antisymmetric on the set of integer homology 3-spheres. Suppose that $M$ and $M'$ are two integer homology 3-spheres such that $M\succeq_1 M'$ and $M'\succeq_1 M$. Since the fundamental groups of manifolds in $\mathcal{M}$ are residually finite (see \cite{He2} for Haken manifolds, the general case can be proved similarly based on Thurston's geometrization conjecture), they are hopfian groups. Then since degree one maps induce epimorphisms between fundamental groups, $\pi_1(M)$ and $\pi_1(M')$ are isomorphic.

Let $f:M\rightarrow M'$ denote the degree one map and $f_*:\pi_1(M)\rightarrow\pi_1(M')$ be the induced map. By the prime decomposition theorem of 3-manifolds, there exists a compact 3-manifold $N_0$ embedded in $M$ such that $\partial N_0$ consists of spheres, $\widehat{N}_0$ is homeomorphic to $S^3$, and for each component $N$ of $\overline{M-N_0}$ the 3-manifold $\widehat{N}$ is an irreducible integer homology 3-sphere. Suppose that $\overline{M-N_0}$ has $m$ components, denoted by $N_1,N_2,\cdots,N_m$. Let $M_i$ be $\widehat{N}_i$, where $0\leq i\leq m$. Similarly for $M'$ we can have $N_j'$ and $M_j'$, where $0\leq j\leq m'$. By the positive solution of Poincar\'e conjecture (see \cite{Pe}), we can assume that neither $\pi_1(M_i)$ nor $\pi_1(M_j')$ is trivial for $i,j\geq 1$. We can further assume that the base points of $M$ and $M'$ lie in $N_0$ and $N_0'$ respectively, and $f$ preserves the base points. Then $\pi_1(M_i)$ and $\pi_1(M_j')$ can be embedded in $\pi_1(M)$ and $\pi_1(M')$ according to the decompositions given by $N_i$ and $N_j'$ respectively. By the Grushko decomposition theorem (see \cite{He1}), $m=m'$ and there exists a permutation $\sigma$ of $\{1,2,\cdots,m\}$ such that $f_*(\pi_1(M_i))$ and $\pi_1(M_{\sigma\cdot i}')$ are conjugate in $\pi_1(M')$ for each $1\leq i \leq m$.

By the positive solution of Thurston's geometrization conjecture and Borel conjecture in dimension three, an irreducible integer homology 3-sphere is determined by its fundamental group. Moreover, all these 3-manifolds are aspherical, except $S^3$ and the Poincar\'e homology sphere $S^3_P$. Hence $M_i$ and $M_{\sigma\cdot i}'$ are homeomorphic for each $1\leq i \leq m$. Note that $M_i$ and $M_{\sigma\cdot i}'$ have the induced orientations from $M$ and $M'$ respectively. We need to show that the homeomorphisms between $M_i$ and $M_{\sigma\cdot i}'$ can be chosen to preserve the orientations. Then as connected sums the oriented 3-manifolds $M$ and $M'$ are homeomorphic.

Fix $1\leq i\leq m$, there is a degree one map $f_i':M'\rightarrow M_{\sigma\cdot i}'$ obtained by pinching each $N_j'$ with $j\neq 0,\sigma\cdot i$ to a point. Since $\pi_2(M_{\sigma\cdot i}')$ is trivial, the degree one map $f_i'\circ f$ induces a family of maps $f_{ji}:M_j\rightarrow M_{\sigma\cdot i}'$ where $0\leq j\leq m$, and we have the equality $\sum_{j=0}^m\deg f_{ji}=1$, where ``$\deg$'' denotes ``degree''. If $M_{\sigma\cdot i}'$ is aspherical, then $f_{ji}$ with $j\neq i$ are all null-homotopic, and we have $\deg f_{ii}=1$. If $M_{\sigma\cdot i}'$ is $S^3_P$, then $120\mid \deg f_{ii}-1$, and there exists a degree one map from $M_i$ to $M_{\sigma\cdot i}'$. In each case, the degree one map is homotopic to a homeomorphism (see \cite{Su}).
\end{proof}

\begin{example}\label{exa:lens space}
Neither $\succeq_1$ nor $\succeq$ is antisymmetric on the set of lens spaces.

Let $N'$ be the product $S^1\times A$, where $A$ is an oriented annulus whose boundary has the induced orientation. Together with an orientation of $S^1$ we get orientations of $N'$ and $\partial N'$. Given integers $a_1$, $b_1$, $a_2$, $b_2$ such that all the greatest common divisors $(a_1,b_1)$, $(a_2,b_2)$, $(a_1,a_2)$ are equal to $1$, a closed 3-manifold $M$ can be obtained from $N'$ by filling solid tori into the boundary tori, mapping the meridians to the curves of slopes $b_1/a_1$ and $b_2/a_2$ respectively.

Choose a disk $D$ in the interior of $A$, whose boundary has the induced orientation from $D$. Let $N$ be the product $S^1\times D$. Then the curve of slope $(a_1b_2+a_2b_1)/(a_1a_2)$ in $\partial N$ bounds a surface in $\overline{M-N}$. A closed 3-manifold $M'$ can be obtained from $N$ by filling a solid torus into $\partial N$, mapping the meridian to this slope.

Note that $(a_1b_2+a_2b_1,a_1a_2)=1$, $M$ is the lens space $L(a_1b_2+a_2b_1,ma_2+nb_2)$ where $m$ and $n$ are integers satisfying the equality $ma_1-nb_1=1$, and $M'$ is the lens space $L(a_1b_2+a_2b_1,a_1a_2)$. The above constructions show that $M\succeq_1 M'$ and $M\succeq M'$. Below we will show that if integers $p$, $q_1$, $q_2$ satisfy $(p,q_1)=(p,q_2)=1$ and $q_1q_2\equiv r^2\pmod p$ for some integer $r$, then there exist $a_1$, $b_1$, $a_2$, $b_2$, $m$, $n$ such that $M$ is $L(p,q_1)$ and $M'$ is $L(p,q_2)$.

Since $(p,q_1)=(p,q_2)=1$, we have $(p,r)=1$ and there exists an integer $q_1^*$ such that $q_1q_1^*\equiv 1\pmod p$. Let $a_1=rq_1^*+p$, $a_2=r$, then $(a_1,a_2)=1$. Hence there exist integers $b_1'$ and $b_2'$ such that $a_1b_2'+a_2b_1'=1$. Let $b_1=b_1'p$, $b_2=b_2'p$, then we have $(a_1,b_1)=(a_2,b_2)=1$ and $a_1b_2+a_2b_1=p$. Clearly $a_1a_2\equiv r^2q_1^*\equiv q_2\pmod p$. Finally, since $(a_1,b_1)=1$, there exist integers $m$ and $n$ such that $ma_1-nb_1=1$, and we have $ma_2+nb_2\equiv mr\equiv ma_1q_1\equiv q_1\pmod p$.
\end{example}

\begin{definition}\label{def:splittable}
A 3-manifold $M$ is {\it $\Sigma_g$-splittable} if every embedded $\Sigma_g$ in $M$ separates $M$ into two parts.
\end{definition}

\begin{proposition}\label{pro:split order}
For $M$ in $\mathcal{M}$, $M$ is $\Sigma_g$-splittable for $g\leq k$ if $M$ is $\Sigma_k$-splittable; $M$ is $\Sigma_g$-splittable for any $g$ if and only if $M$ is a rational homology 3-sphere.
\end{proposition}

\begin{proof}
If an embedded $\Sigma_g$ does not separate $M$, then locally add 1-handles to $\Sigma_g$. Clearly $M$ is not $\Sigma_k$-splittable for $k\geq g$. $M$ is not a rational homology 3-sphere if and only if there is an epimorphism $H_1(M,\mathbb{Z})\rightarrow H_1(S^1,\mathbb{Z})$, which can always be induced by a map from $M$ to $S^1$. The latter condition is equivalent to that $M$ contains a closed two sided surface which does not separate $M$ (see \cite{He1}).
\end{proof}

\begin{proposition}\label{pro:sum}
A 3-manifold $M$ in $\mathcal{M}$ is $\Sigma_g$-splittable if and only if every prime factor of $M$ is $\Sigma_g$-splittable.
\end{proposition}

\begin{proof}
When $g=0$, this is equivalent to that $M$ does not contain $S^1\times S^2$ as a prime factor. Assume that it is true for $g=k-1$. When $g=k$, we only need to show the ``if'' part. By induction, we can assume that the embedded $\Sigma_k$ in $M$ is incompressible. For a sphere in $M$ intersecting $\Sigma_k$ transversely, an innermost disk in the sphere together with a disk in $\Sigma_k$ will form a sphere separating $M$. Then we can remove the intersection by an isotopy or reduce the problem to a 3-manifold with fewer prime factors. Hence the proof can be finished by induction.
\end{proof}

\begin{example}\label{exa:splittable}
Any spherical 3-manifold is a rational homology 3-sphere, hence is $\Sigma_g$-splittable for any $g$. Any hyperbolic 3-manifold in $\mathcal{M}$ contains no incompressible torus, hence is $\Sigma_1$-splittable. Any irreducible 3-manifold is $\Sigma_0$-splittable. $S^1\times S^2$ is not $\Sigma_g$-splittable for any $g$. The mapping torus of a hyperelliptic involution of $\Sigma_g$ is not $\Sigma_g$-splittable, but it is $\Sigma_k$-splittable for any $k<g$.
\end{example}


\section{Extension of periodic automorphisms}\label{sec:extending}

In this section, we prove our main result Theorem \ref{thm:stronger}. By Proposition \ref{pro:split order}, it implies Theorem \ref{thm:general case}. We also discuss several generalizations of Theorem~\ref{thm:stronger}, summarized as Theorem~\ref{thm:compact case} and Theorem~\ref{thm:periodic case}.

\begin{theorem}\label{thm:stronger}
Given a $\Sigma_1$-splittable $M$ in $\mathcal{M}$ and a finite subgroup $G$ in $Aut(\Sigma_g)$, if there is an embedding $e:\Sigma_g\rightarrow M$ such that each element of $G$ is extendable over $M$ with respect to $e$, then there exists $M'$ in $\mathcal{M}$ such that $M\succeq M'$ and $G$ is extendable over $M'$.
\end{theorem}

Its proof uses similar ideas as \cite{Bo1}, as well as \cite{Fl}. According to the sphere and torus/annulus decompositions of $M-\Sigma_g$, we can change complicated 3-manifolds embedded in $M-\Sigma_g$ into 3-balls or solid tori such that $G$ is extendable, and each replacement corresponds to a relation $\succeq$. We first list several fundamental results in the 3-manifold theory that will be used in the proof.

In what follows, Lemma \ref{lem:sphere dec} is based on Kneser-Milnor's sphere decomposition theorem (see \cite[Lemma~A.1]{Bo1}). Theorem \ref{thm:disk dec} can be found in \cite{Bo1}. Theorem \ref{thm:torus dec} is based on the JSJ-decomposition theorem and Thurston's hyperbolizaion theorem (see \cite{Bo2}). Lemmas~\ref{lem:extend comp} and~\ref{lem:extend bundle} can be found in \cite[Propositions~4.1~and~4.3]{Bo1} for cyclic $G$ and closed $\partial_I Z$. Theorem \ref{thm:extend hyp} is based on Mostow's hyperbolic rigidity theorem and Waldhausen's isotopy theorem (see \cite{Bo2} and \cite{Wa}).

\begin{lemma}\label{lem:sphere dec}
Let $X$ be a compact connected oriented 3-manifold with $\partial X\neq\emptyset$. As in Proposition \ref{pro:degree one}, there exists a collection of disjoint spheres $\mathcal{S}$ in $X-\partial X$, which decomposes $X$ into a punctured $S^3$ and several one-punctured prime factors. Suppose that $\mathcal{S}$ and $\mathcal{S}'$ are two such collections, then there exists an element $f$ in $Aut(X)$ such that $f(\mathcal{S})=\mathcal{S}'$ and $f$ fixes $\partial X$.
\end{lemma}

\begin{theorem}\label{thm:disk dec}
Let $X$ be an irreducible compact connected oriented 3-manifold with $\partial X\neq\emptyset$. Then, up to isotopy, $X$ contains a unique compression body $V$ such that the external boundary $\partial_e V$ is equal to $\partial X$, the closure $\overline{X-V}$ contains no essential compression disk for its boundary $\partial_i V$, and no component of $\overline{X-V}$ is a 3-ball.
\end{theorem}

\begin{theorem}\label{thm:torus dec}
Let $Y$ be a compact connected oriented 3-manifold with $\partial Y\neq\emptyset$, which is irreducible and boundary irreducible. Then, up to isotopy, there is a unique minimal collection of disjoint properly embedded essential tori and annuli $\mathcal{T}$ in $Y$ such that for every piece $Z$ obtained by cutting $Y$ along $\mathcal{T}$, either

(i) $Z$ is an $I$-bundle over a compact surface, where ``$I$'' denotes the unit interval, such that the corresponding $\partial I$-bundle is equal to $Z\cap\partial Y$, or

(ii) $Z$ is a Seifert manifold such that $Z\cap\partial Y$ is a union of fibers, or

(iii) Let $Z^*$ be the manifold obtained by removing the tori in $\partial Z$ and the annuli in $\partial Z\cap\mathcal{T}$, then $Z^*$ admits a complete hyperbolic structure with totally geodesic boundary and with finite volume, and the removed torus and annuli in $\partial Z$ are limits of ends of $Z^*$.
\end{theorem}

\begin{lemma}\label{lem:extend comp}
Let $V$ be a connected oriented compression body. For a finite subgroup $G$ in $Aut(\partial_e V)$, if each element $h$ of $G$ is extendable over $V$, then the corresponding element $h'$ in $Aut(V)$ can be deformed to an element $\phi(h)$ in $Aut(V)$ by an isotopy fixing $\partial_e V$ such that the map $\phi:G\rightarrow Aut(V)$ is a group monomorphism.
\end{lemma}

\begin{lemma}\label{lem:extend bundle}
Let $Z$ be a connected oriented $I$-bundle over a compact surface with negative Euler characteristic. Let $\partial_I Z$ denote the $\partial I$-bundle. For a finite subgroup $G$ in $Aut(\partial_I Z)$, if each element $h$ of $G$ is extendable over $Z$, then the corresponding element $h'$ in $Aut(Z)$ can be deformed to an element $\phi(h)$ in $Aut(Z)$ by an isotopy fixing $\partial_I Z$ such that the map $\phi:G\rightarrow Aut(Z)$ is a group monomorphism.
\end{lemma}

\begin{theorem}\label{thm:extend hyp}
Let $Z^*$ be a connected oriented complete hyperbolic 3-manifold with (possibly empty) totally geodesic boundary and with finite volume. (Note that the finite volume condition implies that the ends of $Z^*$ are annular or toral.) Then each element in $Aut(Z^*)$ is isotopic to a unique isometry in $Aut(Z^*)$. Moreover, the isometry must have finite order.
\end{theorem}

We sketch proofs of Lemmas \ref{lem:extend comp} and \ref{lem:extend bundle} according to \cite{Bo1}.

\begin{proof}[Proof of Lemma \ref{lem:extend comp}]
By classical results in differential topology, Lemma \ref{lem:extend comp} holds when $V$ is a 3-ball or an $I$-bundle over $\partial_e V$. Below we assume that $\partial_e V$ has genus bigger than one and contains an essential simple closed curve that bounds a disk in $V$. As in \cite{Th}, we can choose a $G$-invariant hyperbolic structure on $\partial_e V$. Then among the essential curves that bound disks in $V$ there exists a curve $c$ having the minimum length. Since each element $h$ of $G$ is extendable over $V$, the curve $c$ will satisfy $h(c)=c$ or $h(c)\cap c=\emptyset$ for any $h$ in $G$. Let $\{c_1,\cdots,c_m\}$ be the orbit of $c$, $h_i$ be an element in $G$ such that $h_i(c)=c_i$, and $\{D_1,\cdots,D_m\}$ be a collection of embedded disjoint disks in $V$ such that $\partial D_i=c_i$ for $1\leq i\leq m$.

Since $V$ is irreducible, each extension $h'$ in $Aut(V)$ can be deformed to preserve the collection of disks $\mathcal D$ by an isotopy fixing $\partial_e V$. We can deform $h'$ extending the $G$-action onto a regular neighborhood $U$ of the union of $\partial_e V$ and $\mathcal D$. On $\overline{V-U}$ we can use results about 3-balls and $I$-bundles, or repeat the above procedure. The proof can be finished by induction on the genus of the external boundary.
\end{proof}

\begin{proof}[Proof of Lemma \ref{lem:extend bundle}]
Let $\Sigma$ denote the 0-section of the $I$-bundle.
We first consider the case that $\Sigma$ is closed and orientable, then $Z=\Sigma\times[0,1]$. We will follow the argument in \cite{Bo1}.
For each element $g\in G$, let $\rho_i(g): \Sigma\to \Sigma$ be the map induced by the restriction of $g$ on $\Sigma\times\{i\}, i=0,1$. Then $\rho_0,\rho_1$ are two representations of $G$ into $Aut(\Sigma)$.
Since $G$ is finite, we can equip $\Sigma$ with a conformal structure $m_i$ for which the $\rho_i(G)$ action is conformal, $ i=0,1$. By Teichm\"uller theory, given any homeomorphism $f$ of $\Sigma$, there is a unique quasi-conformal map $\tau_f$ from the Riemann surface $(\Sigma,m_0)$ to the Riemann surface $(\Sigma,m_1)$ with constant dilatation, such that $\tau_f$ is isotopic to $f$. Since $m_i$ is preserved by $\rho_i(g)$, and since $\rho_0(g)$ and $\rho_1(g)$ are isotopic, it follows from the uniqueness of the Teichm\"uller mapping $\tau_f$ that
\[
\tau_{\mathrm{id}}\rho_0(g)=\tau_{\rho_0(g)}=\tau_{\rho_1(g)}=\rho_1(g)\tau_{\mathrm{id}}.
\]
Thus $\tau_{\mathrm{id}}$ realizes a conjugacy from $\rho_0(g)$ to $\rho_1(g)$ for each $g\in G$, and is isotopic to the identity. Let $R_t:\Sigma\to \Sigma$ be the isotopy with $R_0=\mathrm{id}$, $R_1=\tau_{\mathrm{id}}$. We define
\[
\phi(g)(x,t)=(R_t\rho_0(g)R_t^{-1})(x), \quad x\in\Sigma, \quad t\in[0,1],
\]
then $\phi: G\to Aut(Z)$ is what we need.

If $\Sigma$ is non-orientable, $Z$ is the orientable twisted $I$--bundle over $\Sigma$, and $Z$ is doubly covered by $\overline Z=(\partial Z)\times[0,1]$. Each $h\in G$ lifts to two maps $h_1,h_2\in Aut(\partial\overline Z)$. Let $\widetilde G<Aut(\partial\overline Z)$ be the group consisting of all such lifts, then every element in $\widetilde G$ extends over $\overline Z$ by lifting the corresponding $h'$. We can apply the case in the last paragraph to get a monomorphism $\widetilde G\to Aut(\overline Z)$. Modulo the covering involution, we can get our $\phi$.

If $\partial\Sigma$ is nonempty, then $\partial Z$ is the union of $\partial_I Z$ and the $I$-bundle over $\partial\Sigma$. We consider the double of $Z$ along this $I$-bundle. This leads to the previous case.
\end{proof}

\begin{lemma}\label{lem:extend hyp}
Let $Z$ be a compact orientable $3$--manifold, $\partial_0Z$ be the union of the torus components of $\partial Z$, $A\subset \partial_1 Z=(\partial Z)\setminus(\partial_0Z)$ be a collection of annuli. Suppose that $Z^*=Z\setminus (A\cup\partial_0Z)$ has a complete hyperbolic metric with (possibly empty) totally geodesic boundary and with finite volume.
Let $F$ be $(\partial_1 Z)\setminus A$ or $\partial_1Z$, and
let $G$ be a finite subgroup of $Aut(F)$, such that each element $h$ of $G$ is extendable over $Z$. Then there exists a group monomorphism $\phi:G\rightarrow Aut(Z)$ such that the restriction of $\phi(h)$ on $F$ is $h$.
\end{lemma}
\begin{proof}
We present $Z$ as a union $Z^{\#}\cup_{F\times\{1\}}F\times[0,1]$, where $Z^{\#}$ is homeomorphic to $Z$, and $F\subset \partial Z$ is identified with $F\times\{0\}$.
Given $h\in G$, let $h'\in Aut(Z)$ be its extension over $Z$. By Theorem~\ref{thm:extend hyp}, there exists a unique isometry $\psi(h)$ of $Z^*$ in the isotopy class of $h'$. Let $\psi^{\#}(h)$ be the corresponding automorphism of $Z^{\#}$, then $\psi^{\#}$ preserves $F\times\{1\}$.
Since $\psi(h)$ is isotopic to $h'$, there exists a map $h''\in Aut(F\times[0,1])$, such that $h''|F\times\{0\}=h$ and $h''|F\times\{1\}=\psi^{\#}(h)|F\times\{1\}$. By Lemma~\ref{lem:extend bundle}, we can extend the $G$--action on $F\times\{0,1\}$ over $F\times[0,1]$, thus we get a $G$--action over $Z$ which extends $G<Aut(F)$.
\end{proof}



\begin{proof}[Proof of Theorem \ref{thm:stronger}]
The case when $g=0$ is clear. Below we assume that $g\geq 1$.

The surface $e(\Sigma_g)$ is two-sided in $M$. We first cut $M$ along $e(\Sigma_g)$ and choose a connected component of the resulting manifold, denoted by $X$, which is clearly $\Sigma_1$-splittible. Since $G$ preserves the two sides of $e(\Sigma_g)$, $G$ naturally embeds into $Aut(\partial X)$ and every element in $G$ is extendable over $X$. For each such component $X$, we will modify it as in the paragraph after Lemma~\ref{lem:h-l h-d}, such that $G$ is extendable over it. The manifold $M'$ will be obtained when the boundaries of the modified components are glued back to $e(\Sigma_g)$. Then $G$ is extendable over $M'$ which will satisfy $M\succeq M'$.

If $X$ admits a nontrivial prime decomposition, then by Lemma \ref{lem:sphere dec}, for each $h$ in $G$ there exists an element $h'$ in $Aut(X)$ such that $h'\circ e=e\circ h$ and $h'$ preserves the decomposition spheres. Note that $\partial X$ lies in a single prime factor, because $M$ is $\Sigma_0$-splittable. We replace $X$ by the prime factor containing $\partial X$, which is irreducible, and still use $X$ to denote it. The new $X$ is still $\Sigma_1$-splittible.

Then by Theorem \ref{thm:disk dec}, we have a compression body $V$ in $X$. For each $h$ in $G$ we can assume that its corresponding element $h'$ in $Aut(X)$ preserves $V$. Note that $V$ is connected if and only if $\partial_e V=\partial X$ is connected, and $G$ acts on each component of $\partial_e V$. Then by Lemma \ref{lem:extend comp}, we can further deform $h'$ such that the map given by $h\mapsto h'|_V$ from $G$ to $Aut(V)$ is a group monomorphism. Namely $G$ is extendable over $V$. We need to extend the action of $G$ further onto  $\overline{X-V}$. 

Let $Y$ be a component of $\overline{X-V}$, and let $G_Y$ be the stable subgroup of $\partial Y$.
Since $M$ is $\Sigma_1$-splittible, one can prove the following property: 
\begin{equation}\label{eq:TorusInY}
\begin{array}{c}
\text{Every torus in the interior of $Y$ splits $Y$ into two parts,} \\
\text{such that $\partial Y$ is contained in one part.}
\end{array}
\end{equation}

\noindent{\bf Claim.} There exists a (possibly empty) collection $\mathcal T^*=T_1\cup\cdots\cup T_n$ of disjoint tori in $Y$, such that 
the components of
$Y\setminus \mathcal T^*$ are $Y_0, Y_1\dots Y_n$ with $\partial Y\subset Y_0$ and $T_i=\partial Y_i$.
Moreover, there exists a monomorphism $\phi: G_Y\to Aut(Y_0)$, such that $\phi(h)|(\partial Y)=h$ for each $h\in G_Y$, and $\phi(h)$ extends to an automorphism of $Y$.

If the above claim holds, we can replace each $Y_i$, $i>1$, by a solid torus as in Definition~\ref{def:partial order}. Since $\phi(h)$ extends to an automorphism of $Y$ for each $h$,
$\phi(G_Y)$ is extendable over the collection of the new solid tori. Then our theorem holds.
In the rest of the proof, we will prove the above claim.

By (\ref{eq:TorusInY}), either $\partial Y$ is a torus or every connected component of $\partial Y$ has genus bigger than one. If $\partial Y$ is a torus, let $\mathcal T^*$ consist of one parallel copy of $\partial Y$, and the claim is obvious.

Below we assume that every connected component of $\partial Y$ has genus bigger than one. By Theorem~\ref{thm:disk dec}, $Y$ is irreducible and boundary irreducible. By Theorem~\ref{thm:torus dec}, there is a collection of tori and annuli $\mathcal{T}$ cutting $Y$ into pieces of certain types.

First of all, suppose that $\mathcal{T}$ contains no annulus. For each $h$ in $G_Y$ we can assume that its corresponding element $h'$ in $Aut(Y)$ preserves $\mathcal{T}$. By (\ref{eq:TorusInY}), $\partial Y$ belongs to a single piece $Z$ of $Y\setminus\mathcal{T}$, which can not be a Seifert manifold. If $Z$ is an $I$-bundle over a closed surface, then it equals $Y$, and $G_Y$ is extendable over it by Lemma \ref{lem:extend bundle}. Otherwise, after removing the boundary tori, $Z$ admits a complete hyperbolic structure with totally geodesic boundary and with finite volume. By Lemma~\ref{lem:extend hyp}, $G_Y$ is extendable over $Z$. 
By (\ref{eq:TorusInY}), every component $T$ of $\mathcal T$ bounds another piece $Z'$. 
We can let $Y_0=Z$, and let $\mathcal T^*=(\partial Y_0)\setminus(\partial Y)$. Then our claim is proved in this case.

Secondly, suppose that $\mathcal{T}$ contains annuli. Let $\mathcal{A}$ be the union of the annuli. Then $\partial\mathcal{A}$ is a collection of disjoint embedded essential circles in $\partial Y$, and for each $h$ in $G_Y$ the two collections $\partial\mathcal{A}$ and $h(\partial\mathcal{A})$ are isotopic. Hence the annuli in $\mathcal{T}$ can be chosen such that $\partial\mathcal{A}$ is preserved by $G_Y$, which can be shown by using a $G_Y$-invariant hyperbolic structure on $\partial Y$. Then for each $h$ in $G_Y$ we can assume that its corresponding element $h'$ in $Aut(Y)$ preserves $\mathcal{T}$. 

Let $\mathcal P$ be the collection of the pieces in $Y\setminus\mathcal{T}$ which are $I$--bundles over surfaces of negative Euler characteristic, and $\mathcal H$ be the collection of the pieces in $Y\setminus\mathcal{T}$ which are hyperbolic with nonempty totally geodedic boundary. Suppose $Z\in\mathcal P\cup \mathcal H$, then $\partial Z$ contains a component of genus $>1$, which  must intersect $\partial Y$. Let $G_Z$ be the stable subgroup of $Z\cap\partial Y$. 

If $Z\in \mathcal P$, by Lemma~\ref{lem:extend bundle}, $G_Z$ is extendable over $Z$, and the extension of each $h\in G_Z$ over $Z$ is unique up to isotopy. 
Moreover, suppose $A\subset\mathcal T$ is an annulus, then on at most one side of $A$ there is a piece in $\mathcal P$, otherwise we could extend the $I$--bundle structures across $A$ to get bigger $I$--bundles. Let $Y_P$ be the union of all pieces in $\mathcal P$ and a tubular neighborhood of $\partial Y$, then $G_Y$ is extendable over $Y_P$. 

Let $Y'=\overline{Y\setminus Y_P}$, and let $\mathcal A'$ be the union of annuli in $\mathcal A$ which are not in $Y_P$. Let $A$ be a component of $\mathcal A'$, and let $Z\in\mathcal H$ be  adjacent to $A$. By Lemma~\ref{lem:extend hyp}, $G_Z$ can be extended over $Z$. In particular, $G_Z$ can be extended over the union of the annuli bounded by the $G_Z$--orbit of $\partial A$. In this way, we can inductively extend $G_Y$ over $\mathcal A'$. 

Suppose $Z'$ is a piece of $\mathcal H$ and $T$ is a torus component of $\partial Z'$, then $T$ does not contain any component of $\mathcal A'$. Otherwise, $Z'$ would have a totally geodesic annulus boundary comoponent, which is impossible. Hence $T$ is not parallel to any torus in $(\partial Y)\cup\mathcal A'$, on which $G_Y$ now acts.
Using Lemma~\ref{lem:extend hyp} again, we can extend $G_Y$ over all pieces of $\mathcal H$. Now let $Y_0$ be the union of the pieces in $\mathcal P\cup\mathcal H$ and a neighborhood of $\partial Y$, and let $T^*=(\partial Y_0)\setminus(\partial Y)$. The claim is thus proved.
\end{proof}

In what follows, we discuss some generalizations of Theorem \ref{thm:stronger}.

1. We can replace $\Sigma_g$ by any compact connected 3-manifold $N$. The extendable automorphisms and extendable subgroups in $Aut(N)$ can be defined similarly as the surface case. The proof is almost the same, except that $X$ should be a connected component of $\overline{M-N}$ now, and we consider the stable subgroup of $\partial X$ in $G$.

We can also replace $\Sigma_g$ by any compact connected surface. Let $\Sigma$ denote such an oriented surface. We can first extend the action of $G$ onto an oriented $I$-bundle over $\Sigma$ embedded in $M$, then extend the action of $G$ onto a regular neighborhood of $\Sigma$. This leads to the case of compact 3-manifolds.

For a compact connected nonorientable surface $\Pi$, let $Aut(\Pi)$ denote the group of automorphisms of $\Pi$, then we can define extendable automorphisms and extendable subgroups in $Aut(\Pi)$ similarly. Since $M$ is oriented, the surface $e(\Pi)$ is one-sided. In $M$ we can choose an oriented $I$-bundle $Z$ over $\Pi$, and assume that for each $h$ in $G$ the corresponding element $h'$ preserves $Z$ and the bundle structure. Note that the $\partial I$-bundle $\partial_I Z$ is the two-sheeted orientable cover of $\Pi$. Each $h$ has two lifts in $Aut(\partial_I Z)$, and they differ by the covering transformation. Because $h'$ preserves the two sides of $\partial_I Z$, its restriction on $\partial_I Z$ is the orientation-preserving lift of $h$. Then the action of $G$ can extend onto a regular neighborhood of $\Pi$.

Clearly we can also replace $\Sigma_g$ by $S^1$ or $I$. Hence we have the following result.

\begin{theorem}\label{thm:compact case}
Given a $\Sigma_1$-splittable $M$ in $\mathcal{M}$, a compact connected manifold $N$, and a finite subgroup $G$ in $Aut(N)$, if there is an embedding $e: N\rightarrow M$ such that each element of $G$ is extendable over $M$ with respect to $e$, then there exists $M'$ in $\mathcal{M}$ such that $M\succeq M'$ and $G$ is extendable over $M'$.
\end{theorem}

2. Theorem \ref{thm:compact case} will still hold if we add orientation-reversing automorphisms into $Aut(N)$ and $Aut(M)$ when we define ``extendable''. Note that the results listed after Theorem \ref{thm:stronger} are still valid when we add orientation-reversing elements into the groups. In the case of compact 3-manifolds, the proof is also valid. In the cases of compact surfaces, we need more discussions.

For the oriented surface $\Sigma$, we can define a map $\rho: G\rightarrow\mathbb{Z}_2$ such that $\rho(h)=0$ if and only if the element $h'$ corresponding to $h$ preserves the two sides of $e(\Sigma)$. For the nonorientable surface $\Pi$, we can define a map $\rho: G\rightarrow\mathbb{Z}_2$ such that $\rho(h)=0$ if and only if the element $h'$ corresponding to $h$ preserves the orientation of $M$. Note that for $\Pi$ we have an oriented $I$-bundle $Z$ over it, $h'$ is orientation-preserving on $M$ if and only if its restriction on $\partial_I Z$ is orientation-preserving.

If $\rho$ is a group homomorphism, then the action of $G$ can extend onto a regular neighborhood of $e(\Sigma)$ (resp. $e(\Pi)$). This leads to the case of compact 3-manifolds. Otherwise, the identity on $\Sigma$ (resp. $\Pi$) can extend to an automorphism of $M$ that exchanges the two sides (resp. reverses the orientation). Then every element in $G$ can extend to an automorphism of $M$ that preserves the two sides (resp. preserves the orientation), and the proof can be finished as the previous case.

For periodic automorphisms, similar arguments can show the following theorem, where when $N$ is a surface, it is possible that $f$ has odd order while $f'$ reverses the orientation of $M$, in which case $h$ has order twice the order of $f$.

\begin{theorem}\label{thm:periodic case}
If a periodic automorphism $f$ of a compact connected manifold $N$ extends to an automorphism $f'$ of a $\Sigma_1$-splittable $M$ in $\mathcal{M}$, then there exists $M'$ in $\mathcal{M}$ such that $M\succeq M'$, $f$ can extend to a periodic automorphism $h$ of $M'$, and $h$ preserves the orientation of $M'$ if and only if $f'$ preserves the orientation of $M$.
\end{theorem}


\section{Classification of extendable cyclic actions}\label{sec:cyclic}
In this section, we first prove Theorem \ref{thm:cyclic classify}. According to the theory of orbifolds, which is developed and carefully discussed in \cite{Th} and \cite{BMP}, two subgroups $G$ and $G'$ in $Aut(\Sigma_g)$ are conjugate to each other by an automorphism $f$ of $\Sigma_g$ if and only if there is an isomorphism $\eta:G\rightarrow G'$ and a homeomorphism $\tau:\Sigma_g/G\rightarrow\Sigma_g/{G'}$ such that $\eta\circ\psi=\psi'\circ\tau_*$, where $\psi$ and $\psi'$ are the epimorphisms corresponding to $G$ and $G'$ respectively and $\tau_*$ is the induced homomorphism $\pi_1(\Sigma_g/G)\rightarrow\pi_1(\Sigma_g/{G'})$. Moreover, $f$ is orientation-preserving if and only if $\tau$ is orientation-preserving.

We will write the group law in the cyclic group $\mathbb Z_m$ multiplicatively. In particular, the identity element in $\mathbb Z_m$ will be $1$. The following lemma is an easy consequence of the Chinese remainder theorem. We leave its proof to the reader.

\begin{lemma}\label{lem:cyc group}
Let $m,p,q$ be positive integers, where $(p,q)=1$. Let $h$ be a generator of the group $\mathbb{Z}_{mpq}$. Then for any elements $a,b$ in $\mathbb{Z}_{mpq}$ with orders $p,q$ respectively, there is an automorphism $\eta$ of $\mathbb{Z}_{mpq}$ such that $\eta(a)=h^{mq}$ and $\eta(b)=h^{mp}$.
\end{lemma}

\begin{lemma}\label{lem:orb group}
Let $\mathcal{F}$ be the 2-orbifold having underlying space $\Sigma_r$ and $s$ singular points of indices $n_1,\cdots,n_s$. According to Figure \ref{fig:presentation}, $\pi_1(\mathcal{F})$ has the presentation
\[\langle \alpha_1,\beta_1,\cdots,\alpha_r,\beta_r,\gamma_1,\cdots,\gamma_s\mid \prod^r_{i=1}[\alpha_i,\beta_i]\prod^{s}_{j=1}\gamma_j=1,\gamma_k^{n_k}=1,1\leq k\leq s\rangle.\]

\begin{figure}[h]
\includegraphics{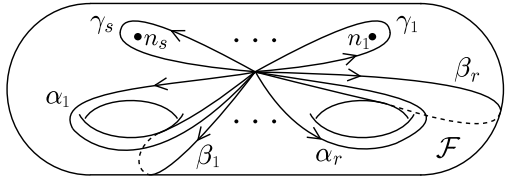}
\caption{Generators of $\pi_1(\mathcal{F})$}\label{fig:presentation}
\end{figure}

Let $\psi:\pi_1(\mathcal{F})\rightarrow\mathbb{Z}_m$ be a finitely-injective epimorphism, and $h$ be a generator of $\mathbb{Z}_m$. Then there is an automorphism $\tau$ of $\mathcal{F}$ such that $\psi\circ\tau_*(\alpha_1)=h$, $\psi\circ\tau_*(\alpha_i)=1$ for $2\leq i\leq r$, $\psi\circ\tau_*(\beta_i)=1$ for $1\leq i\leq r$, and $\psi\circ\tau_*(\gamma_j)=\psi(\gamma_j)$ for $1\leq j\leq s$.
\end{lemma}

\begin{proof}
If $r=0$, then let $\tau$ be the identity. Below we will assume that $r\geq 1$. Note that we do not need to consider the base point, because $\mathbb{Z}_m$ is abelian. In fact, $\psi$ factors through $H_1(\mathcal F)$, so we often abuse the notation by regarding $\psi$ as a map $H_1(\mathcal F)\rightarrow\mathbb{Z}_m$. We will also use $\alpha_i$, $\beta_i$, $\gamma_j$, $1\leq i\leq r$, $1\leq j\leq s$ to denote the loops presenting them.

Let $n$ be the least common multiple of $n_1,\cdots,n_s$. Since $\psi$ is injective on finite subgroups of $\pi_1(\mathcal{F})$, the subgroup of $\mathbb{Z}_m$ generated by $\psi(\gamma_1),\cdots,\psi(\gamma_s)$ has order $n$. Below we will consider slides and Dehn twists on $\mathcal{F}$ along the loops $\alpha_i$ and $\beta_i$ for $1\leq i\leq r$. Figure \ref{fig:slides} shows the sketch of slides, where we omit the indices of the singular points. The left two pictures indicate the slide of singular points, and the right two pictures indicate the slide of handles.

\begin{figure}[h]
\includegraphics{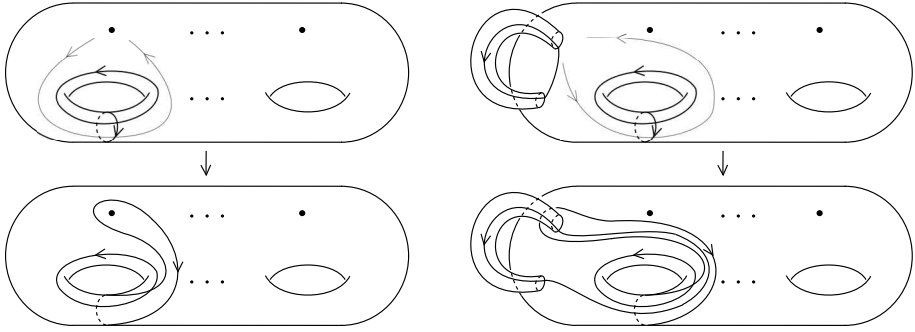}
\caption{Slide of singular points and slide of handles}\label{fig:slides}
\end{figure}

If $m=n$, then $\psi$ is surjective on $\langle\gamma_1,\cdots,\gamma_s\rangle$. Each singular point of $\mathcal{F}$ can slide along the loop $\alpha_i$ or $\beta_i$. Consider the singular point corresponding to $\gamma_k$. The slide of it along $\alpha_i$ will change $\psi(\beta_i)$ to $\psi(\beta_i\gamma_k^{\pm 1})$ and not affect the $\psi(\beta_j)$ with $j\neq i$, $\psi(\alpha_j)$ with $1\leq j\leq r$ and $\psi(\gamma_j)$ with $1\leq j\leq s$. Similarly, the slide of it along $\beta_i$ will only affect $\psi(\alpha_i)$. Hence we can change $\psi(\alpha_i)$ and $\psi(\beta_i)$ for $1\leq i\leq r$ to the required values, and $\tau$ can be a composition of slides of singular points.

If $m>n$, then we also need slides of handles and Dehn twists. The Dehn twist along $\alpha_i$ will change $\psi(\beta_i)$ to $\psi(\beta_i\alpha_i^{\pm 1})$ and not affect the $\psi(\beta_j)$ with $j\neq i$, $\psi(\alpha_j)$ with $1\leq j\leq r$ and $\psi(\gamma_j)$ with $1\leq j\leq s$. Similarly, the Dehn twist along $\beta_i$ will only affect $\psi(\alpha_i)$. Hence we can change $\psi(\alpha_i)$ for $1\leq i\leq r$ to $1$ by Dehn twists. Each handle corresponds to $\{\alpha_k,\beta_k\}$ can slide along the loop $\alpha_i$ where $i\neq k$. The slide will change $\psi(\beta_i)$ to $\psi(\beta_i\beta_k^{\pm 1})$ and not affect the $\psi(\beta_j)$ with $j\neq i$, $\psi(\alpha_j)$ with $j\neq k$ and $\psi(\gamma_j)$ with $1\leq j\leq s$. It will also not affect $\psi(\alpha_k)$ if each $\psi(\alpha_j)$ is $1$. So we can further change $\psi(\beta_i)$ for $2\leq i\leq r$ to $1$ by slides of handles. Now $\psi(\beta_1)$ is a generator of $\mathbb Z_m$ since $m>n$. We can change $\psi(\alpha_1)$ to the given generator $h$ by Dehn twists along $\beta_1$, then change $\psi(\beta_1)$ to $1$ by Dehn twists along $\alpha_1$.
Hence by the Dehn twists, the slides of handles, and the slides of singular points, we can change $\psi(\alpha_i)$ and $\psi(\beta_i)$ for $1\leq i\leq r$ to the required values, and $\tau$ can be a composition of slides of singular points, slides of handles and Dehn twists.
\end{proof}

\begin{proof}[Proof of Theorem \ref{thm:cyclic classify}]
We first show the ``only if'' part. Suppose that $G$ is a finite cyclic subgroup in $Aut(\Sigma_g)$ that is extendable over $S^3$. By the results in \cite{Pe}, we can assume that the monomorphism $\phi$ maps $G$ into $SO(4)$. Then we can obtain an embedding of a 2-orbifold $\mathcal{F}=\Sigma_g/G$ in the spherical 3-orbifold $\mathcal{O}=S^3/\phi(G)$.

We can identify $S^3$ with the set $\{(z_1,z_2)\in\mathbb{C}^2\mid |z_1|^2+|z_2|^2=1\}$. Let $h$ be a generator of $G$. Then we can assume that $\phi(h)$ has the form
\[\phi(h):(z_1,z_2)\mapsto(e^{i\theta_1}z_1,e^{i\theta_2}z_2).\]
Suppose that $G$ has order $n$, then there exist integers $l_1$ and $l_2$ such that $\theta_1=2\pi l_1/n$ and $\theta_2=2\pi l_2/n$. Note that the greatest common divisor of $l_1,l_2,n$ is $1$, otherwise the order of $G$ is smaller that $n$. Assume that neither $z_1$ nor $z_2$ is zero and there exists a positive integer $k$ such that $\phi(h)^k((z_1,z_2))=(z_1,z_2)$. Then $n\mid k$, and this means that the image of $(z_1,z_2)$ in $\mathcal{O}$ is a regular point. The image of $(z_1,0)$ (resp. $(0,z_2)$) in $\mathcal{O}$ has index $(l_1,n)$ (resp. $(l_2,n)$). Hence it is singular if and only if the corresponding greatest common divisor is bigger than one.

Let $p=(l_1,n)$ and $q=(l_2,n)$, then $(p,q)=1$. The singular set of $\mathcal{O}$ consists of at most two circles with indices $p$ and $q$. Clearly $n_1,\cdots,n_s\in\{p,q\}$, and each of $p$ and $q$ is valued even times, because $\mathcal{F}$ separates $\mathcal{O}$. Then the singular points of $\mathcal{F}$ in each singular circle can be partitioned into pairs such that the conditions (b) and (c) in Theorem \ref{thm:cyclic classify} are satisfied.

Then we show the ``if'' part. Suppose the orbifold $\mathcal{F}=\Sigma_g/G$ and the finitely-injective epimorphism $\psi$ satisfy the three conditions (a), (b), (c) in Theorem \ref{thm:cyclic classify}. By (a), each $\psi(\gamma_k)$ is either an element of order $p$ or an element of order $q$. By (b), all the $\psi(\gamma_k)$ with order $p$ (resp. $q$) have at most two possible values. Then by (c) and Lemma \ref{lem:cyc group}, the values of $\psi(\gamma_k)$ can be determined, up to automorphisms of $G$. Note that there exist automorphisms of $\mathcal{F}$ that exchange any two singular points in $\mathcal{F}$ with the same index. Then combined with Lemma \ref{lem:orb group}, up to conjugacy, we can assume that $\psi$ satisfies the following conditions.

(1) $\psi(\alpha_1)=h$, $\psi(\alpha_i)=1$ for $2\leq i\leq r$, $\psi(\beta_i)=1$ for $1\leq i\leq r$;

(2) $\psi(\gamma_j)=h^{mq}$ for $1\leq j\leq s_1$, $\psi(\gamma_j)=h^{-mq}$ for $s_1+1\leq j\leq 2s_1$, $\psi(\gamma_j)=h^{mp}$ for $2s_1+1\leq j\leq 2s_1+s_2$, $\psi(\gamma_j)=h^{-mp}$ for $2s_1+s_2+1\leq j\leq s$.

Here $h$ is a generator of $G$; $s_1, s_2\geq 0$ are two integers such that $2s_1+2s_2=s$. Since $\psi$ is finitely-injective, we have an integer $m$ such that $mpq$ equals the order of $G$, where $p=1$ (resp. $q=1$) if $s_1=0$ (resp. $s_2=0$).

Hence, when $\mathcal{F}$ is a torus with no singular points, up to conjugacy, the $G$-action on $\Sigma_1$ is determined by the order of $G$, and it is extendable over $S^3$; otherwise, up to conjugacy, the $G$-action on $\Sigma_g$ is completely determined by $\Sigma_g/G$, and we need to show that such a $G$-action is extendable over $S^3$.

For any 2-orbifold $\mathcal{F}$ and finitely-injective epimorphism $\psi:\pi_1(\mathcal{F})\rightarrow\mathbb{Z}_n$ that satisfy conditions (a), (b) and (c), we will construct an extendable cyclic action of order $n$ in Example \ref{exa:standard form} such that its corresponding orbifold is $\mathcal{F}$. By the ``only if'' part, this new action also satisfies the three conditions, and hence is conjugate to the $G$-action by the above arguments. Since in Example \ref{exa:standard form} all the surfaces in $S^3$ are Heegaard surfaces, we have the ``moreover'' part of Theorem \ref{thm:cyclic classify}.
\end{proof}

\begin{example}\label{exa:standard form}
Let $\mathcal{F}$ be the 2-orbifold having underlying space $\Sigma_r$, $2s_1$ singular points of index $p$, and $2s_2$ singular points of index $q$, where $s_1,s_2$ are nonnegative integers, $p,q$ are co-prime positive integers, and $s_1=0$ (resp. $s_2=0$) if and only if $p=1$ (resp. $q=1$). There exists a finitely-injective epimorphism $\psi:\pi_1(\mathcal{F})\rightarrow\mathbb{Z}_n$ if and only if one of the following conditions holds.

(1) $r=0$ and $n=pq$;

(2) $r\geq 1$ and $n=mpq$ for some positive integer $m$.

In each case, we will construct a $\mathbb{Z}_n$-action on a closed surface below such that its corresponding orbifold is homeomorphic to $\mathcal{F}$. The action is extendable over $S^3$ and its corresponding embedded surface is a Heegaard surface.

We identify $S^3$ with the set $\{(z_1,z_2)\in\mathbb{C}^2\mid |z_1|^2+|z_2|^2=1\}$. Let $\theta_m^q=2\pi/mq$ and $\theta_m^p=2\pi/mp$. Then let $h_m^{p,q}$ be the isometry of $S^3$ defined by
\[h_m^{p,q}:(z_1,z_2)\mapsto(e^{i\theta_m^q}z_1,e^{i\theta_m^p}z_2).\]
Clearly $h_m^{p,q}$ has order $mpq$. Let $T$ be the set $\{(z_1,z_2)\in S^3\mid |z_1|=|z_2|\}$. Then $T$ is a torus in $S^3$, which is invariant under the action of $\langle h_m^{p,q}\rangle$. Let $\epsilon$ be a sufficiently small positive number, for example $\epsilon<(100n(s_1+1)(s_2+1))^{-1}$.

\begin{figure}[h]
\includegraphics{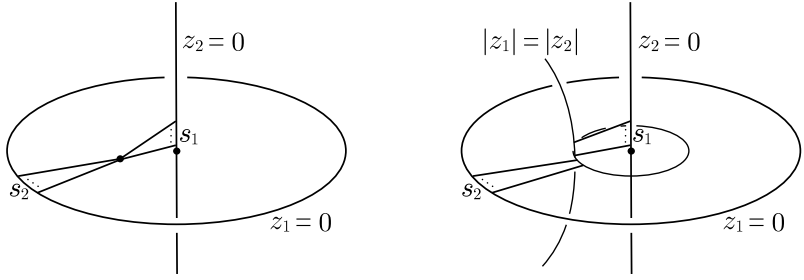}
\caption{Part of $\Gamma_{s_1,s_2}^{p,q}$ and $\Gamma_{s_1,s_2,m}^{p,q}$ in $S^3-\{(-1,0)\}$.}\label{fig:Sdform}
\end{figure}

In the case of (1), we connect the point $(\sqrt{2}/2,\sqrt{2}/2)$ to the points $(e^{ik\epsilon},0)$ and $(0,e^{il\epsilon})$ for $1\leq k\leq s_1$ and $1\leq l\leq s_2$ by the shortest geodesics in $S^3$. The left picture in Figure \ref{fig:Sdform} shows a sketch of the construction. The union of the orbits of these geodesics under the action of $\langle h_1^{p,q}\rangle$ forms a graph $\Gamma_{s_1,s_2}^{p,q}$. Then we can choose an invariant regular neighborhood of $\Gamma_{s_1,s_2}^{p,q}$. Denote its boundary by $\Sigma_{s_1,s_2}^{p,q}$. Then the orbifold $\Sigma_{s_1,s_2}^{p,q}/\langle h_1^{p,q}\rangle$ is homeomorphic to $\mathcal{F}$.

In the case of (2), we connect $(e^{ik\epsilon}\sqrt{2}/2,\sqrt{2}/2)$ to $(e^{ik\epsilon},0)$ for $1\leq k\leq s_1$, and connect $(\sqrt{2}/2,e^{il\epsilon}\sqrt{2}/2)$ to $(0,e^{il\epsilon})$ for $1\leq l\leq s_2$, by the shortest geodesics in $S^3$. The right picture in Figure \ref{fig:Sdform} shows a sketch of the construction. The union of the orbits of these geodesics under the action of $\langle h_m^{p,q}\rangle$ forms a graph $\Gamma_{s_1,s_2,m}^{p,q}$, which is not connected in general. We choose sufficiently small disjoint spheres centred at the vertices of degree bigger than one. For each edge of the graph we can make a tube along it connecting the small spheres to the torus $T$. Then we can obtain a closed surface which is invariant under the action of $\langle h_m^{p,q}\rangle$. At each point in the orbit of $(\sqrt{2}/2,\sqrt{2}/2)$ we can add $r-1$ local handles to the surface equivariantly. Denote the result by $\Sigma_{s_1,s_2,m}^{p,q,r}$. Then $\Sigma_{s_1,s_2,m}^{p,q,r}/\langle h_m^{p,q}\rangle$ is homeomorphic to $\mathcal{F}$.

Note that the surfaces $\Sigma_{s_1,s_2}^{p,q}$ and $\Sigma_{s_1,s_2,m}^{p,q,r}$ are all Heegaard surfaces. Hence the $\langle h_1^{p,q}\rangle$-action on $\Sigma_{s_1,s_2}^{p,q}$ and the $\langle h_m^{p,q}\rangle$-action on $\Sigma_{s_1,s_2,m}^{p,q,r}$ give all extendable finite cyclic actions on $\Sigma_g$ stated in Theorem \ref{thm:cyclic classify}. By the Riemann-Hurwitz formula,
\[2-2g=n(2-2r-2s_1(1-\frac{1}{p})-2s_2(1-\frac{1}{q})).\]
Since $(p,q)=1$ and $pq\mid n$, for a given $g$ we can find all solutions $(n,p,q,r,s_1,s_2)$ by enumeration. Especially, when $s_1=s_2=0$ we obtain the standard forms of the free finite cyclic actions on $\Sigma_g$, which correspond to the factors of $g-1$.
\end{example}

\begin{remark}\label{rem:equivalent}
Note that for a finite cyclic subgroup $G$ in $Aut(\Sigma_g)$ that is extendable over $S^3$, the generators of $G$ are not conjugate to each other in general. Let $\psi$ be the finitely-injective epimorphism. For two generators $h_1$ and $h_2$ in $G$, let $\eta$ be the automorphism of $G$ such that $\eta(h_1)=h_2$. Then $h_1$ and $h_2$ are conjugate to each other if and only if $\eta\circ\psi(\gamma_k)=\psi(\gamma_k)$ or $\eta\circ\psi(\gamma_k)=\psi(\gamma_k)^{-1}$ for $1\leq k\leq s$. Hence the classification of orientation-preserving periodic automorphisms of $\Sigma_g$ that are extendable over $S^3$ needs a little more work.
\end{remark}

As the end of the paper, we give some examples and questions as illustrations and supplements to our theorems.

\begin{example}\label{exa:figure eight}
Let $Z$ be the complement of an open regular neighborhood of the figure eight knot in $S^3$. Then $\partial Z$ is a torus, which admits a translation of arbitrarily large order. On the other hand, $Z-\partial Z$ admits a complete hyperbolic structure with finite volume (see \cite{Th}). The orders of periodic automorphisms of it are bounded. Clearly a translation of $\partial Z$ is extendable over $Z$. But if its order is large enough, then the group generated by it is not extendable over $Z$.

\begin{figure}[h]
\includegraphics{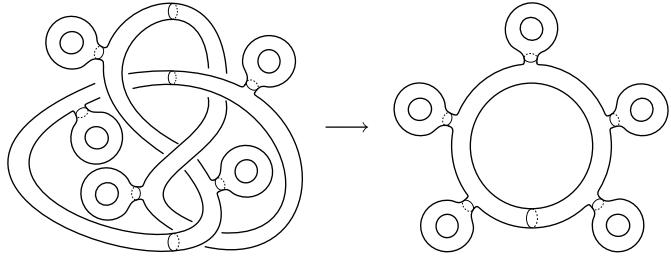}
\caption{Replace the complement by a simple manifold}\label{fig:figure8}
\end{figure}

Let $G$ be such a group. Then under the above embedding of $\partial Z$ in $S^3$, $G$ is not extendable. However, we can replace $Z$ by a solid torus such that the result manifold is still a $S^3$ and $G$ is extendable over it. Namely there is another embedding of $\partial Z$ in $S^3$ such that $G$ is extendable over $S^3$.

Examples about surfaces of high genera can be obtained by adding handles onto the closed regular neighborhood of the figure eight knot, as in Figure \ref{fig:figure8}. The surface constructed in this way is compressible in each sides.
\end{example}

\begin{example}\label{exa:non extendable}
We give two periodic automorphisms of the handlebody of genus four, which are not extendable over $S^3$. However, it is not so easy to prove this fact without Theorem \ref{thm:sphere case} and Theorem \ref{thm:cyclic classify}.

\begin{figure}[h]
\includegraphics{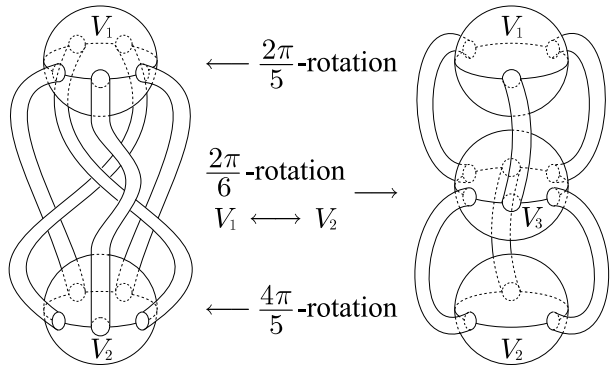}
\caption{Actions on the handlebody of genus four}\label{fig:handlebody}
\end{figure}

We first consider the handlebody with two 0-handles $V_1,V_2$ and five 1-handles $E_i$ for $1\leq i\leq 5$. We will construct a periodic automorphism $f$ of order five. It is a $2\pi/5$-rotation on $V_1$, is a $4\pi/5$-rotation on $V_2$, and permutes the five 1-handles such that $f(E_i)=E_{i+1}$ for $1\leq i\leq 4$. Then we can equivariantly attach the five 1-handles onto $V_1$ and $V_2$ such that each $E_i$ is adjacent to both of $V_1$ and $V_2$. See the left picture in Figure \ref{fig:handlebody}. This gives an order five periodic automorphism of the handlebody of genus four. Intuitively, it is not extendable over $S^3$, because the ``rotation speeds'' of the two 0-handles are different. In view of our theorems, the condition (b) in Theorem \ref{thm:cyclic classify} is not satisfied.

Then consider the handlebody with three 0-handles $V_1,V_2,V_3$ and six 1-handles $E_i$ for $1\leq i\leq 6$. We will construct a periodic automorphism $f$ of order six. It is a $2\pi/6$-rotation on $V_3$ and it exchanges $V_1$ and $V_2$. Its square preserves $V_1$ and $V_2$, and is a $2\pi/3$-rotation on them. It permutes the six 1-handles such that $f(E_i)=E_{i+1}$ for $1\leq i\leq 5$. Then we can equivariantly attach the six 1-handles onto $V_1$, $V_2$, $V_3$ such that each of $E_1$, $E_3$, $E_5$ is adjacent to both of $V_1$ and $V_3$, and each of $E_2$, $E_4$, $E_6$ is adjacent to both of $V_2$ and $V_3$. See the right picture in Figure \ref{fig:handlebody}. This gives an order six periodic automorphism of the handlebody of genus four. Intuitively, it is not extendable over $S^3$, because $V_1$ and $V_2$ can not be exchanged. In the view of our theorems, the condition (a) in Theorem \ref{thm:cyclic classify} is not satisfied.
\end{example}

\begin{example}\label{exa:knotted surface}
By the result in \cite{WWZZ2}, the maximum order of finite subgroups in $Aut(\Sigma_g)$ that are extendable over $S^3$ is $6(g-1)$ when $g$ is $21$ or $481$. Moreover, the corresponding embedded surfaces of such groups cannot be Heegaard surfaces. Hence, the ``moreover'' part of Theorem \ref{thm:cyclic classify} does not hold for general finite group actions. Below we will give an explicit example about $\Sigma_{21}$.

Consider two spheres in $\mathbb{R}^3$ centered at the origin with radii $1/2$ and $2$. Project a regular dodecahedron centered at the origin onto the spheres, and denote the two images by $P_1$ and $P_2$, where $P_2$ is the larger one. Clearly the orientation-preserving isometries of the dodecahedron preserves $P_1\cup P_2$. Consider the composition of the inversion about the unit sphere and a refection about a plane containing the origin. We can choose the plane such that the composition preserves $P_1\cup P_2$. Then all these elements preserving $P_1\cup P_2$ generate a group of order $120$, denoted by $G$.

Clearly $G$ can act on $S^3$. Let $v$ and $w$ be adjacent vertices in the dodecahedron, and let $v_1$, $w_1$ and $v_2$, $w_2$ be their corresponding vertices in $P_1$ and $P_2$. Then we can choose an arc connecting $v_1$ and $w_2$ such that its images under the action of $G$ only meet at vertices of $P_1$ and $P_2$. The union of these image arcs is a graph with genus $21$. It has a regular neighborhood $N$ that is preserved by the action of $G$. Then $G$ also preserves $\partial N$ which is homeomorphic to $\Sigma_{21}$.
\end{example}

\begin{question}
If we admit orientation-reversing automorphisms (of $\Sigma_g$ or $S^3$), how to classify the periodic automorphisms of $\Sigma_g$ that is extendable over $S^3$? can the corresponding embedded surface always be a Heegaard surface?
\end{question}

\bibliographystyle{amsalpha}

\begin{thebibliography}{WWZZ1}
\bibitem[BMP]{BMP}
M. Boileau, S. Maillot, J. Porti, {\it Three-dimensional orbifolds and their geometric structures}, Panoramas et Synth\`eses 15. Soci\'et\'e Math\'ematique de France, Paris, 2003.

\bibitem[BRW]{BRW}
M. Boileau, J. H. Rubinstein, S. C. Wang, {\it Finiteness of 3-manifolds associated with non-zero degree mappings}, Comment. Math. Helv. 89 (2014), no. 1, 33-68.

\bibitem[Bo1]{Bo1}
F. Bonahon, {\it Cobordism of automorphisms of surfaces}, Ann. Sci. \'Ecole Norm. Sup. (4)16 (1983), no. 2, 237-270.

\bibitem[Bo2]{Bo2}
F. Bonahon, {\it Geometric structures on 3-manifolds}, Handbook of geometric topology, 93-164, North-Holland, Amsterdam, 2002.


\bibitem[ES]{ES}
J. Eells, J. H. Sampson, {\it Harmonic mappings of Riemannian manifolds}, Amer. J. Math. 86 (1964), 109-160.

\bibitem[Fl]{Fl}
E. Flapan, {\it Rigidity of graph symmetries in the 3-sphere}, J. Knot Theory Ramifications 4 (1995), no. 3, 373-388.

\bibitem[FY]{FY}
E. Flapan, S. Yu, {\it Symmetries of Spatial Graphs in Homology Spheres}, arXiv:1907.03130.

\bibitem[FK]{FK}
K. Funayoshi, Y. Koda, {\it Extending automorphisms of the genus-2 surface over the 3-sphere}, to appear in  Q. J. Math., arXiv:1803.05116.

\bibitem[GWWZ]{GWWZ}
Y. Guo, C. Wang, S. C. Wang, Y. M. Zhang, {\it Embedding periodic maps on surfaces into those on $S^3$}, Chin. Ann. Math. Ser. B 36 (2015), no. 2, 161-180.


\bibitem[Ha]{Ha}
A. E. Hatcher, {\it A proof of the Smale conjecture, $Diff(S^3)\simeq O(4)$}, Ann. of Math. (2) 117 (1983), no. 3, 553-607.

\bibitem[He1]{He1}
J. Hempel, {\it 3-manifolds}, Ann. of Math. Studies, No.86. Princeton University Press, Princeton, N.J.; University of Tokyo Press, Tokyo, 1976.

\bibitem[He2]{He2}
J. Hempel, {\it Residual finiteness for 3-manifolds}, Combinatorial group theory and topology (Alta, Utah, 1984), 379-396, Ann. of Math. Stud., 111, Princeton Univ. Press, Princeton, N.J., 1987.



\bibitem[Liu]{Liu}
Y. Liu, {\it Nonzero degree maps between three dimensional manifolds}, to appear in J. Topol., arXiv:1107.5855.

\bibitem[Pe]{Pe}
G. Perelman, {\it The entropy formula for the Ricci flow and its geometric applications} (eprints arXiv: math/0211159); {\it Ricci flow with surgery on three-manifolds} (eprints arXiv: math/0303109); {\it Finite extinction time for the solutions to the Ricci flow on certain three-manifolds} (eprints arXiv: math/0307245).


\bibitem[RZ]{RZ}
M. Reni, B. Zimmermann, {Extending finite group actions from surfaces to handlebodies}, Proc. Amer. Math. Soc. 124 (1996), no. 9, 2877-2887.

\bibitem[Su]{Su}
H. B. Sun, {\it Degree $\pm 1$ self-maps and self-homeomorphisms on prime 3-manifolds}, Algebr. Geom. Topol. 10 (2010), no. 2, 867-890.

\bibitem[Th]{Th}
W. P. Thurston, {\it The geometry and topology of three-manifolds}, Lecture notes, 1978.

\bibitem[Wa]{Wa}
F. Waldhausen, {\it On irreducible 3-manifolds which are sufficiently large}, Ann. of Math. (2) 87 (1968), 56-88.

\bibitem[WWZZ1]{WWZZ1}
C. Wang, S. C. Wang, Y. M. Zhang, B. Zimmermann, {\it Extending finite group actions on surfaces over $S^3$}, Topology Appl. 160 (2013), no. 16, 2088-2103.

\bibitem[WWZZ]{WWZZ2}
C. Wang, S. C. Wang, Y. M. Zhang, B. Zimmermann, {\it Embedding surfaces into $S^3$ with maximum symmetry}, Groups Geom. Dyn. 9 (2015), no. 4, 1001-1045.

\end{thebibliography}

\end{document}